\documentclass[12pt,oneside,a4paper,mathscr]{amsart}
\usepackage{amsthm,amsmath,amssymb,euscript,amscd,a4wide}
\usepackage[all]{xy}
\usepackage{multirow,bigstrut}
\usepackage{tikz}

\newtheorem{thm}{Theorem}[section]

\newtheorem{lemma}[thm]{Lemma}
\newtheorem{prop}[thm]{Proposition}

\theoremstyle{definition}

\newtheorem{note}[thm]{Notes}

\theoremstyle{remark}

{\end{note}}

\numberwithin{equation}{section}

\newcommand{\C}{{\mathbb C}}

\newcommand{\R}{\mathbb{R}}
\newcommand{\PP}{{\mathbb P}}
\newcommand{\OO}{{\mathscr O}}

\newcommand{\A}{\mathcal{A}}
\newcommand{\EE}{\mathbb{E}}

\newcommand{\Hilb}{\operatorname{Hilb}}
\newcommand{\coh}{\operatorname{Coh}}

\newcommand{\Pic}{\operatorname{Pic}} 
\newcommand{\Ext}{\operatorname{Ext}}

\newcommand{\Hom}{\operatorname{Hom}}
\newcommand{\rhom}{\operatorname{\textbf{R}\textit{Hom}}}

\newcommand{\coker}{\operatorname{coker}}
\newcommand{\tors}{\operatorname{tors}}

\newcommand{\ch}{\operatorname{ch}}
\newcommand{\rk}{\operatorname{rk}}

\newcommand{\M}{{\mathcal M}}
\newcommand{\T}{{\mathbb T}}
\renewcommand{\P}{{\mathscr P}}

\newcommand{\dT}{{\hat{\mathbb T}}}

\newcommand{\xhat}{{\Hat x}}
\newcommand{\yhat}{{\Hat y}}

\newcommand{\I}{{\mathscr I}}

\newcommand{\compo}{\raise2pt\hbox{$\scriptscriptstyle\circ$}}
\renewcommand{\leq}{\leqslant}
\renewcommand{\geq}{\geqslant}

\newcommand{\tof}[1]{\mathbin{\buildrel{#1}\over{\to}}}

\def\rep#1{\bysame}

\begin{document}

\title[Bridgeland Stable Moduli Spaces]{Rank One Bridgeland Stable Moduli Spaces on A
  Principally Polarized Abelian Surface}
\author[]{Antony Maciocia \&\ Ciaran Meachan}
\address{Department of Mathematics and Statistics\\
The University of Edinburgh\\
The King's Buildings\\ Mayfield Road\\ Edinburgh, EH9 3JZ.\\}
\email{A.Maciocia@.ed.ac.uk\\C.P.Meachan@ed.ac.uk}
\thanks{} 
\date{\today}
\begin{abstract}
We compute moduli spaces of Bridgeland stable objects on an
irreducible principally polarized complex abelian surface $(\T,\ell)$
corresponding to twisted ideal sheaves. We use Fourier-Mukai
techniques to extend the ideas of Arcara and Bertram to express
wall-crossings as Mukai flops and show that the moduli spaces are projective.
\end{abstract}

\maketitle

\section*{Introduction}{}
Let $(\T,\ell)$ be a principally polarized abelian surface over $\mathbb{C}$. We shall
assume that $\Pic\T=\langle\ell\rangle$. We shall also fix a line
bundle $L$ with $c_1(L)=\ell$. Then the linear system $|\ell|$
consists of a unique smooth divisor $D$ given as the zero set of the
unique (up to scale) section of $L$. We can translate $D$ to give a
family of divisors which we shall denote by $D_x=\tau_x D\in
|\tau^*_{-x}\ell|$. As observed in \cite{Mac11}, we can view these
$D_x$ as analogues of lines on the projective plane. They have the
property that any two intersect in exactly 2 points (up to
multiplicity) and any two points (or fat point) lies on exactly two of
them. Given a 0-scheme $X\subset\T$ we will say that $X$ is
\emph{collinear} if there is some $x$ such that $X\subset D_x$. 

Now consider objects of the (bounded) derived
category $D(\T)$ whose Chern characters are $(1,2\ell,4-n)$, for an
integer $n\geq0$. A torsion-free sheaf with such a Chern character
takes the form $L^2\otimes\P_\xhat\otimes\I_X$, where $\P_\xhat$
is the flat line bundle corresponding to some point  $\xhat$ of the
dual torus $\dT$ and $X$ is a 0-scheme of length $n$. We shall drop
the tensor product signs in what follows. The Gieseker moduli space of
such objects is a fine moduli space given by $\Hilb^n\T\times\dT$. We
can view this asymptotically as a Bridgeland stable moduli space (see
\cite{BrK3}) in a certain abelian subcategory $\A_0$ (defined below). There is a
1-parameter family of stability conditions indexed by a positive real
number $t$ in $\A_0$. For some large $t_0$, if $t>t_0$ the moduli
functor
\begin{equation}\label{e:modfunct}
\M^{(1,2\ell,4-n)}_t:\Sigma\mapsto\left\{a\in D^b(\T\times
  \Sigma):\begin{matrix}\ch(a_\sigma)=(1,2\ell,4-n),\\ a_\sigma\text{ is
      $t$-stable, }
\text{for all $\sigma\in \Sigma$, } \end{matrix}\right\}/\sim,\end{equation}
where $i_\sigma:\T\to\T\times \Sigma$ is the inclusion corresponding
to $\sigma\in \Sigma$, $a_\sigma=Li^*a_\sigma$ and $\sim$ is the usual
equivalence relation $a\sim a\otimes\pi^*_\Sigma M$, for any line bundle
$M$ on $\Sigma$, 
is represented by $\Hilb^n\T\times \dT$. We shall omit the superscript
on $\mathcal{M}$ if the context is clear. As $t$ decreases, we expect to cross
walls as some of the objects become unstable. The object of this paper
is to describe the resulting moduli spaces $M_t$ for all $t>0$ and
all $n\geq0$. In fact for $n<3$, the torsion-free sheaves are
$t$-stable for all $t>0$.  When $n>3$ there is more than one moduli
space and each moduli space is modified by a
Mukai flop as we cross the wall. This is very similar to what is
found in Arcara and Bertram \cite{Arcara07} where the case $\ch(a)=(0,H,H^2/2)$ is
computed in the abelian category $\A_{H/2}$ (there $H$ is some
polarization). In that case, the 
situation is made complicated by the presence of ``higher rank
walls'' and a complete picture is not given. In our case, there is
only one higher rank wall (when $n=5$) and we can give an explicit
description of that case. Of course our results are much less general
than those of \cite{Arcara07}. We pay this price in order to have a
useful computational tool at our disposal which allows us to be more
explicit in our constructions. However, Arcara and Bertram do prove
that the resulting moduli spaces really do exist as smooth proper
schemes representing \ref{e:modfunct}. We discuss this in section 4.

That tool is the Fourier-Mukai transform. We choose to use the
original such transform defined  by Mukai in \cite{Muk81} (see
\cite{HuyBook} and \cite{BBHBook} for an exposition of the theory),
but shifted by $[1]$ in $D(\dT)$. We shall
denote this by $\Phi:D(\T)\to D(\dT)$. As is now well known, this is
an equivalence of categories. It was used extensively in \cite{Mac11}
to understand how divisors in the linear system $|2\ell|$ intersect and we
shall use several of those computations below. Pulling back the
transform to include a parameter space $\Sigma$ allows us to observe that
$\Phi$ preserves moduli in the sense that if $M$ together with a
universal object $\EE$ represents a moduli functor $\M$ on $\T$ then
$\Phi(M)$ together with $\Phi_\Sigma(\EE)$ represents the pullback functor
$\Phi_\Sigma^*(\M_t)$. But we can improve this using an observation of
Huybrechts (\cite{HuyFMT}). He showed that for any given Fourier-Mukai
transform there is a choice of $\mathbb{R}$-polarizations $\beta$ on
$\T$ and $\beta'$ and $\dT$ such that $\Phi:\A_\beta\stackrel\sim\to\A_{\beta'}$
and moreover $\Phi_\Sigma^*(\M_t)=\Hat\M_{t'}$ for some $t'$ depending on
$t$, $\beta$ and $\beta'$ and where $\Hat\M$ is the same functor as
$\M$ but for $\dT$ and with $\Phi(\ch(a_\sigma))$ instead of
$\ch(a_\sigma)$. In our case, $\beta=0=\beta'$ and $t'=1/t$. We also
have the formula that $\Phi(r,c\ell,\chi)=(-\chi,c\ell,-r)$. We can see from
this why $n=5$ is special for us as that is precisely the case when
$\ch(a)$ is preserved by $\Phi$. Immediately we can conclude that
$\M_t$ is represented by $\Phi(M_{1/t})$ for all $t<1/t_0$ (we shall
see that $t_0=\sqrt{3}$). From
\cite{Mac11} we know that some $L^2\P_\xhat\I_W$, where $|W|=5$ are
not WIT and so there are elements of $M_t$ for small $t$ which are not sheaves.
However, it turns out that this is the only time that non-sheaves can
arise.

We shall see that for any $n$ there are $d=\lfloor\frac{n-1}{2}\rfloor$
walls except when $n=5$ when there is an additional (so called, higher
rank) wall. So there are $\lfloor\frac{n-1}{2}\rfloor+1$ moduli
spaces $M_0$,\ldots, $M_d$ where $M_0$ corresponds to $t\gg 0$. Now
$M_0$ is well known to be given by Gieseker stable sheaves (in this
case, actually $\mu$-stable) and so the usual GIT construction shows
that it is projective. On the other hand, we shall see that
$\Phi(M_d)$ are represented by sheaves as well (so long as $n>3$) and
hence, $M_d$ is also projective. To show that the other spaces $M_i$ are
projective we observe that we can vary $\beta$ and in a suitable range
each moduli space corresponds to a moduli space of Bridgeland stable
objects for $t$ arbitrarily small. Then we can apply a suitable
Fourier-Mukai transform to show that $M_i$ is isomorphic to a
Bridgeland stable moduli space of sheaves but now with $t$
large which are again known to be projective. The difficult step here
is to show that the transforms of the points of $M_i$ are pure
sheaves.  

Finally, we look at the $n=3$, $n=4$ and $n=5$ cases in more
detail. In many ways, the $n=3$ case is the most interesting. There is
a single wall in that case and we show explicitly that the two moduli
spaces are isomorphic. Crossing the wall corresponds to a birational
transformation which replaces a $\mathbb{P}^1$-fibred codimension $1$
subspace with its dual fibration. We will see  explicitly that the resulting
birational map between the two moduli spaces does not extend to an
isomorphism (even though the spaces are actually isomorphic). It also
turns out that for nearby $\A_s$ with $s>0$ there is another wall and
this time it is a codimension $0$ wall.

A more general study of the relation between wall crossing and
Fourier-Mukai transforms is given in \cite{Yosh11}.
\section*{Notation}
\begin{tabbing}
$P$, $Q$, $Y$, $Z$, $W$ \qquad\= 0-schemes of lengths 1, 2, 3, 4 and 5,
respectively\\
$\I_X$\> ideal sheaf of general 0-scheme $X$\\
$L$\> fixed choice of polarizing line bundle with $c_1(L)=\ell$\\
$D(\T)$\> bounded derived category of coherent sheaves on $\T$\\
$(r,c\ell,\chi)$\> typical Chern character of an object
of $D(\T)$\\
$\T\cong\dT$\> canonical identification via our choice of $L$.\\
$a$, $b$, $d$, $e$, \ldots\> arbitrary objects of $D(\T)$\\
$A^i$, $B^i$, $D^i$, $E^i$, \ldots\> cohomology of $a$, $b$, $d$, $e$,\ldots\\
This last piece of notation is to avoid clutter with $H^i({-})$.
\end{tabbing}
\section{Stability Conditions on Abelian Surfaces}
Following Bridgeland \cite{BrK3}, we consider a special collection of
stability conditions on our abelian surface $(\T,\ell)$. These arise
as tilts of $\coh_\T$ and are parametrized by a complex K\"ahler class
$\beta+i\omega$. We will take $\omega=t\ell$ and $\beta=s\ell$. Then
we define a torsion theory by:
\begin{align*}
F_s&=\{E\in\coh_\T:E\text{ is TF, }\mu_+(E)\leq 2s\}\\
T_s&=\{E\in\coh_\T:E\text{ is torsion or }\mu_-(E/\tors(E))>2s\}
\end{align*}
We let the associated tilted abelian subcategories be denoted by
$\A_s$. Explicitly,
\[\A_s=\{a\in D(\T):A^i=0,\ i\neq-1,0,\ A^{-1}\in F_s,\
A^0\in T_s\}\]
(recalling our notational convention that $A^i=H^i(a)$).
This carries a 1-parameter family of stability conditions whose charge
is
\begin{align*} Z_{s,t}(a)&=\langle e^{(s+it)\ell},\ch(a)\rangle\\
&=-\chi+2sc-r(s^2-t^2)+2it(c-rs),
\end{align*}
where $\ch(a)=(r,c\ell,\chi)$. Recall for an abelian surface that the
top part of the Chern character of $a$ is equal to the Euler character
$\chi(a)$.
For a quick proof that this defines a stability condition see
\cite[Cor 2.1]{Arcara07}.
Then $Z_{s,t}$ provides us with a Bridgeland stability condition
on $\A_{s}$. We can then declare an object $a\in\A_{s}$ to be
(Bridgeland) \emph{$t$-stable} provided for each proper subobject
$b\to a$ in $\A_{s}$, we have $\mu_t(b)<\mu_t(a)$, where the
\emph{$t$-slope} $\mu_t(a)$ is given by
\[-\frac{\Re Z_{s,t}(\ch(a))}{\Im
  Z_{s,t}(\ch(a))}=\frac{\chi-2sc+r(s^2-t^2)}{2t(c-rs)}.\]
We view this as taking values in $\mathbb{R}\cup\{\infty\}$, taking an
infinite value precisely when the denominator vanishes.
As an example of how this works we prove the following easy
generalization of \cite[Lemma 3.2]{Arcara07}
\begin{lemma}
If $E$ is a $\mu$-stable torsion-free sheaf which is not locally-free and
$\mu(E)\leq 2s$ then $E[1]\in\A_s$ is not $t$-stable for any $t>0$.
\end{lemma}
\begin{proof}
Observe that we have $\mu_+(E)=\mu(E)<2s$ and so $E\in F_s$. Hence,
$E[1]\in \A_{s}$. But if $X$ is the 0-scheme of the singularity
set of $E$ then we have a sheaf short exact sequence
\[0\to E\to E^{**}\to \OO_X\to 0.\]
Note that $E^{**}$ is still $\mu$-stable and of the same slope as $E$
and so $E^{**}[1]\in\A_s$. Then 
\[0\to \OO_X\to E[1]\to E^{**}[1]\to 0\]
is short exact in $\A_s$. But $\mu_t(\OO_X)=\infty$ and so cannot
be less than $\mu_t(E)$ for any $t$.  
\end{proof}

We also prove
\begin{lemma}
The objects $a$ of $\A_s$ with infinite $t$-slope are given by the
short exact sequence (in $\A_s$)
\[0\to E[1]\to a\to \OO_X\to 0\]
where $X$ is a 0-scheme (possibly empty) and $E$ is a $\mu$-semistable
torsion-free sheaf of slope $2s$ or is the zero sheaf.
\end{lemma}
\begin{proof}
Suppose first that $r(A^0)>0$. Then let $\mu(A^{-1})=c/r$ and
$\mu(A^0)=c'/r'$. For $\mu_t(a)=\infty$ we require $\mu(a)=2s$. But 
$\mu(a)=\frac{c'-c}{r'-r}>\frac{c'}{s}$ because $c'/r'>c/r$ (this is a
characterising property of slope functions: if $x\to y\to z$ is a
short exact sequence then $\mu(y)>\mu(z)$ implies
$\mu(x)>\mu(y)$). But $c'/r'>2s$ as $A^0\in T_s$. This contradiction
implies that $r'=0$. But then $2s=(c'-c)/(-r)=\mu(A^{-1})-c'/r\leq 2s$
with equality only if $\mu(A^{-1})=2s$ and $c'=0$ as required.
\end{proof}

Finally in this section we make the following useful observations
(left as exercises for the reader).
\begin{prop}\label{p:genfact}
\rule{0pt}{0pt}
\begin{enumerate}
\item (Schur's lemma) If $a\in \A_s$ is $t$-stable for some $t>0$ then $\Hom(a,a)$
  consists of automorphisms.
\item If $E\in\A_s\cap\coh_\T$ and there is some $t_0$ such that
  for all $t>t_0$, $E$ is $t$-stable then $E$ must be
  torsion-free. (In fact, any torsion subsheaf must eventually
  $t$-destabilise it).
\item More generally, if $E\in\A_s\cap\coh_\T$ then there is some $t_0$ such
  that for all $t>t_0$, $E$ is $t$-stable if and only if $E$ is (twisted)
  Gieseker stable.
\item If $E[1]\in\A_0\cap\coh_\T[1]$ has $c_1(E)=0$  then $E$
  is $\mu$-semistable.
\item If $E\in \A_0\cap\coh_\T$  satisfies
  $c_1(E)=\ell$ then $E/\tors(E)$ is $\mu$-semistable (or zero).
\end{enumerate}
\end{prop}

For the rest of this paper we will be interested purely in the case
$0\leq s<1$. These have slopes
\[\frac{4-n-4s+s^2-t^2}{2t(2-s)}.\]

Our starting point is the following well known theorem (see for
example \cite{HLBook}) translated into our context:
\begin{thm}
There is some real number $t_0>0$ such that for all $t>0$, $\M^{(1,2\ell,4-n)}_t$ is
represented by the projective space $\Hilb^n\T\times\dT$. A universal sheaf $\EE_t$ is
given by $\pi_1^*L^2\otimes\pi_{13}^*\P\otimes\pi_{12}^*\mathbb{I}_{\mathcal
  Z}$,
where $\P$ is the Poincar\'e bundle over $\T\times\dT$,
$\pi_i$ and $\pi_{ij}$ is the projections from $\T\times\Hilb^n\T\times\dT$ to
the $i^{\mathrm{th}}$ and $ij^{\mathrm{th}}$ factors respectively, and $\I_{\mathcal{Z}}$ is the
ideal sheaf of the tautological universal subscheme
$\mathcal{Z}\subset\T\times\Hilb^n\T$.
\end{thm}

Using Proposition \ref{p:genfact}(3) again and the observation in the
introduction about the Fourier-Mukai transform preserving moduli, we
also have non-empty fine projective moduli spaces $M_t^{(n-4,2\ell,-1)}$ for
$n\geq4$ and so we also see that $\M^{(1,2\ell,4-n)}_t$ is represented
by this space for all $t$ less than some $t_1$.

The situation for $n<3$ is cleared up by the following proposition.
\begin{prop} \label{p:stabs}(see \cite{Arcara09}) The following holds
  in $\A_s$ for all $t>0$ and all $0\leq s<1$.
\begin{enumerate}
\item For all integers $m>0$ $L^m$ is $t$-stable.
\item For all integers $m\leq0$, $L^m[1]$ is $t$-stable.
\item If $E\in \A_s\cap\coh_\T$ has $c_1(E)=\ell$ and $r(E)= 1$ and is
  torsion-free then $E$ is $t$-stable. 
\item For all 0-schemes $X\subset \T$ with $|X|<3$, $L^2I_X$ is
  $t$-stable for $s=0$.
\item If $E\in\A_0$ is a pure torsion sheaf with $c_1(E)=\ell$ then
  $E$ is $t$-stable.
\end{enumerate}
\end{prop}
\begin{proof}
Case (1) is treated in \cite[Proposition 3.6(b)]{Arcara09} but we can
give a more direct proof by observing that a destabilizing object must
be a sheaf $K\to L^m$. Now assume that $K$ is
$\mu$-semistable and $K\to L^m$ is non-zero. If its Chern character is $(r,c\ell,\chi)$ then 
we can re-arrange $\mu_t(K)-\mu_t(L^m)\geq0$ to 
\[t^2\leq\frac{(m-s)(\chi-c(s+m)+mrs)}{rm-c}\]
But $m-s>0$ and $rm-c>0$ (to ensure $\Hom(K,L^m)\neq0$). Now
$r(\chi-c(s+m)+mrs)=(r\chi-c^2)+(c-rs)(c-rm)$. The second term is
negative as $K\in T_s$. The first term is non-positive by
Bogomolov. So every factor $K'$ of the $\mu$-Harder-Narasimhan
filtration of $K$ has $\mu_t(K')<\mu_t(L^m)$ and so $K$ cannot
destabilize $L^m$. (2) is
similar and we leave as an exercise for the reader.

For (3) observe that if $k\to E$ is supposed to destabilize $E$ then
the image of $K^0\to E$ must have slope equal to $E$ and so $Q^0$ is
supported on points, where $q=E/k$ in $\A_s$. But $K^{-1}=0$  and we
are left with a long exact sequence (in $\coh_\T$) 
\[0\to Q^{-1}\to K\to E\to Q^0\to0.\]  
Now assume that $s=0$. Then $\deg(Q^{-1})=\deg(K)-2<0$ since if it equalled to $0$, $q$ would have
infinite slope if $s=0$ and could not destabilize $E$.  But this implies
$\deg(K)=0$ and this can only happen if $K$ is supported in dimension
0, which is impossible as $E$ is torsion-free. This also applies if
$E$ is pure rank $0$ as well and so we have (5) as well.

Returning to (3) with $0<s<1$, we have
just shown that there are no walls intersecting the line $s=0$. Each
wall is a semicircle with centre on the $s$-axis. Let
$\ch(K)=(r,c\ell,\chi)$. Then the destabilizing condition is
\begin{equation}\label{e:destab}\chi(1-s)+s^2(c-r)-\chi(E)(c-rs)>0.\end{equation}
But for $Q^{-1}\in F_s$ we must have $c-1\leq s(r-1)<r-1$. The centre
of the semicircular wall has
\[s=-\frac{1}{2}\frac{\chi-r\chi(E)}{r-c}\]
Since there are no walls at $s=0$ we have $\chi<r\chi(E)$. Then the
destabilizing condition \ref{e:destab} becomes
\[0<-\chi(E)(c-rs-r(1-s))+s^2(c-r)=(-\chi(E)+s^2)(c-r).\]
This is a contradiction unless $\chi(E)=1$. But this is dealt with in (1).

For (4) we proceed as follows (this will be typical of such
proofs). We suppose $L^2\I_X$ is not $t$-stable. Then there must exist
destabilising subobjects $k\to L^2\I_X$. Let the quotient (in $\A_0$)
be $q$ as above. Again $K^{-1}=0$.
Now $K=K^0$ must be torsion-free (because $Q^{-1}\in F_0$) and so has
 positive degree. Let the Chern character of $K$ be $(r,c\ell,\chi)$. 
 Then the
fact that it destabilizes gives us the inequality
\[2\chi+(n-4)c\geq (2r-c)t^2.\]
 But
$\deg(Q^{-1})\leq s<1$ and so $\deg(K/Q^{-1})\geq 2c$. But $r(K/Q^{-1})=1$
and so $\deg(K/Q^{-1})=2$ or $4$. In the latter case, if $c=2$ then
$c_1(q)=0$ but then $q$ cannot destabilize after all. If $c=1$ then
$K$ must be $\mu$-semistable by Prop \ref{p:genfact}(5) and so
$\chi\leq1$ 
by the Bogomolov inequality. But $2\chi\geq 4-n>1$ for $n<3$. So
$\chi=1$ and $n=2$. But this only destabilizes if $t=0$ which is impossible. This
contradiction shows that no such $K$ can exist.
\end{proof}
Note that we only used $n<3$ at the very end so we see more generally
that the only possible destabilising subobject must be a
$\mu$-semistable sheaf of degree $2$. Moreover $\chi\geq 4-n$.

\section{Identifying the Candidate Stable Objects}
Now we look for which objects may be representatives of points of our
moduli spaces. In other words, we find objects $a$ with Chern
character $(1,2\ell,4-n)$ which are $t$-stable for some $t>0$. In this
section we start by assuming assuming $s=0$. 

\begin{prop}\label{p:poss}
Suppose $e\in\A_0$ with $\ch(e)=(1,2\ell,4-n)$ is $t$-stable for some $t>0$. Then, either
\begin{enumerate}
\item $e$ is a torsion-free sheaf $E$, i.e. $E=L^2\I_X\P_{\hat x}$ for
  some $X\in\Hilb^n\T$ and $\hat x\in\dT$, or 
\item $e$ is a sheaf $E$ with torsion, in which case, $\tors(E)$ is
  a line bundle supported on some $D_x$ of degree $4-n+m$ and
  $E/\tors(E)\cong L\I_{X'}\P_{\hat x}$ for some $\xhat\in \dT$ and
  $X'\in\Hilb^m\T$, where $0\leq m <(n-2)/2$, or
\item $e$ is a two-step complex with $E^{-1}\cong L^{-1}\P_{\hat
    x}$ for some $\xhat\in \dT$ and $E^0$ a $\mu$-stable
  locally-free sheaf with $\ch(E^0)=(2,\ell,0)$ only when $n=5$. 
\end{enumerate}
\end{prop}

\begin{proof}
We have already seen that if $e$ is a torsion-free sheaf then it is
$t$-stable for large enough $t$. So we assume that $e$ is not a
torsion-free sheaf.

Now suppose $e$ is a sheaf $E$ with torsion subsheaf $\tors(E)$. Since
sheaves supported on 0-schemes have infinite slope any such subsheaf $S$
of $\tors(E)$ would destabilize $E$ for all $t$ as $E/S\in
T_0$. Observe also that $E$ is not a torsion sheaf and so
$\deg(\tors(E))=2$. Hence, $\tors(E)$ is supported on a translate of
$D$ and locally-free on its support. Suppose it has degree $d$ (so
$\chi(\tors(E))=d-1$). Let $F=E/\tors(E)$. Then
$\ch(F)=(1,\ell,5-n-d)$ and $F$ is torsion-free. So $F\cong
L\P_{\xhat}\I_{X'}$, where $|X|'=n+d-4=m$. Then $d=4-n+m$. But
$\mu_t(\tors(E))=(d-1)/2t$ and this will always destabilize $E$ if
$m\geq(n-2)/2$. So we require $m< (n-2)/2$. Note that such $E$ cannot
be $t$-stable for $t\geq\sqrt{n-2+2m}$ as they are destabilized by their
own torsion.

Now suppose that $e$ is not a sheaf. Let $\ch(E^{-1})=(r,c\ell,\chi)$
with $r\geq1$. Then $\ch(E^0)=(r+1,(2+c)\ell,4-n+\chi)$. then $c<0$
(because if $c=0$, $E^{-1}[1]$ would destabilize $e$ for all $t$). But
$2+c>0$ and so $c=-1$ is the only possible value and $E^{-1}$ must be
$\mu$-semistable. Indeed, if $D$ was a potential $\mu$-destabilising
object then $\deg(D)=0$ and the composite $\A_0$-injection $D[1]\to E^{-1}[1]\to
E$ would destabilise $E$ for all $t>0$; contradiction. Thus, by
Bogomolov, we have $\chi\leq1$ and $E$ is $t$-stable for some
$t>0$ if and only if 
for some $t>0$, $\mu_t(E)<\mu_t(E^0)$ which is equivalent to
$0<(2r+1)t^2<4-n+2$. This implies $n<6$. 

Now let $F=E^0/\tors(E^0)$. Then $c_1(F)=c_1(E)$ as $c_1(E)$ is minimal in $T_0$ and
$\chi(F)=4-n+\chi-p\leq0$ by Bogomolov and Prop
\ref{p:genfact}(5), for some $p\geq0$. But composing $\A_0$-surjections $e\to E^0\to
F$, we see that  there must exist $t$ such that
$\mu_t(F)-\mu_t(e)>0$. This can only happen if
$4-n+2\chi-2p>0$. But $\chi-p\leq n-4$ and so $n-4>0$. Hence, $n=5$
is the only possibility.

When $n=5$ we have $\chi(E^{-1})=\chi=1$ which can only happen if
$r(E^{-1})=1$. Then $E^{-1}\cong L^{-1}\P_{\xhat}$ for some
$\xhat\in\dT$. We also have $s=0$ and $\ch(E)=(2,\ell,0)$. Such a
$\mu$-semistable sheaf must be $\mu$-stable and locally-free.
\end{proof}
So we see that if $n\neq5$, only sheaves can be $t$-stable for some
$t$; all other objects are $t$-unstable for all $t$.

The proposition does not prove that cases (2) and
(3) do actually arise. To show that (3) does arise we use the
Fourier-Mukai transform. Observe that $E^{-1}[1]\to e$ will
destabilize if $t\geq 1/\sqrt{3}$. We now compute the Fourier-Mukai
transform of these objects.
\begin{prop}
Suppose $e\in\A_s$ has $\ch(E^{-1})=(1,-\ell,1)$ and
$\ch(E^0)=(2,\ell,0)$ with $E^0$ torsion-free. Then $\Phi(e)$ is a
torsion-free sheaf. 
\end{prop}
\begin{proof}
We use the spectral sequence $\Phi^{p+q}(e)\Leftarrow \Phi^p(E^q)$.
We have $\Phi(E^{-1})\cong\tau_\xhat^*\hat L[-1]$ (see \cite{Muk81} or
\cite{Mac11}) and $\Phi(E^0)$ is a torsion sheaf of rank 1 supported on some
$D_x$ of degree $-1$. Then the spectral sequence has only two non-zero
terms $E_2^{1,-1}\cong\tau_{\xhat}^*\hat L$ and $E_2^{0,0}\cong
\Phi(E^0)$. So we have a short exact sequence (in $\A_0$):
\[0\to \tau_\xhat^*\hat L\to \Phi(e)\to \Phi(E^0)\to0\]
and so $\Phi(e)$ is in $\A_\T\cap\coh_\T$. To see that it is
torsion-free observe that any torsion must be supported on $D_x$ with
degree less than $-1$. Then $\Phi(e)/\tors(\Phi(e))$ would have Euler
characteristic bigger than 1 which is impossible for a torsion-free
sheaf or rank 1 and degree 2. 
\end{proof}
So $\Phi(e)$ takes the form $\hat L^2\P_x\I_{\hat X}$ for some $\hat
X\in\Hilb^5\dT$ and $x\in\T=\Hat{\dT}$. But for $1/t$ sufficiently
large this is $1/t$-stable and so $e$ is $t$-stable for $t$
sufficiently small. Hence, case (3) does arise (but only if $n=5$).

For case (2), consider a torsion sheaf $G$ supported on $D_x$ of rank
1 and degree $4-n+m$ and some $X'\in\Hilb^m\T$, for some
$m<(n-2)/2$. Observe that
\[\chi(L\I_{X'},G)=1-n+m<-n/2<0\]
and so $\Ext^1(L\I_{X'},G)\neq0$ and hence there are non-trivial extensions
\[0\to G\to E\to L\I_{X'}\to0.\]
$G$ will destabilize $E$ if $t\geq\sqrt{n-2+2m}$. If $t<\sqrt{n-2+2m}$
then we need to check that $E$ can be chosen to be $t$-stable.
 As before there must be a sheaf $K\in T_0$ and an injection $K\to E$
 in $\A_0$ which destabilizes. Let the quotient be $q$. Now both $G$
 and $L\I_{X'}$ are $t$-stable (by Proposition \ref{p:stabs}(3)). We can
 assume that $K$ is itself $t$-stable by picking the first
 Jordan-H\"older co-factor of the first Harder-Narasimhan factor. Then
 $\Hom(K,G)=0$ and so $\Hom(K,L\I_{X'})\neq0$. Note that $\deg(Q^{-1})=0$ as
$r(K/Q^{-1})=1$ and $\Hom(K/Q^{-1},L\I_{X'})\neq0$. But then
$\mu_t(K/Q^{-1})\geq\mu_t(K)$ and $K/Q^{-1}\to E$ injects in $\A_0$
which is impossible given the choice of $K$. So $q=Q^0=Q$ is a sheaf
and $r(K)=1$, $\deg(K)=2$. So $K\cong L\I_{X''}$ for some $X''\supset
X'$. But this can only destabilize for $t^2<n-2-2|X''|$. Hence, for
$\sqrt{n-4-2m}<t<\sqrt{n-2-2m}$, $E$ must be $t$-stable. So
again, case (2) does arise for all $n>2$.

Finally, let us consider the torsion free sheaves of the form
$E=L^2\P_\xhat\I_X$.  The argument at the end of the proof of
Proposition \ref{p:stabs} shows that any destabilizing object of $E$ must
be a torsion-free sheaf of degree $2$. In other words, there is some
0-scheme $X'$ of length $m$ and a map $L\P_\yhat\I_{X'}\to
L^2\P_\xhat\I_X$. As a sheaf map this injects with quotient $G$, a
torsion sheaf of rank $1$ supported on some $D_x$ of degree
$4-n-|X'|$. Now this destabilizes only when $t<\sqrt{n-2-2m}$. 
The existence of such a destabilizing subsheaf can be described
geometrically. The following refines Proposition \ref{p:stabs}(4).
\begin{prop}\label{p:geom}
Let $X$ be a 0-scheme of length $n$. Suppose $X''\subset X$ is a
collinear subscheme of maximal length. Then $E=L^2\P_\xhat\I_X$
  is $t$-stable for all $t>\sqrt{\max(0,2|X''|-n-2)}$. 
\end{prop} 
\begin{proof}
The existence of $X''$ is equivalent to the existence of a non-zero
map from $L\P_\yhat\I_{X'}\to E$ where $X'=X\setminus X''$ and $\yhat$
is some element of $\dT$. The maximality assumption implies that
$|X'|$ is least among such maps and so $E$ is $t$-stable for
all $t^2>n-2-2|X'|=2|X''|-n-2$. 
\end{proof}
Note that the codimension of such loci in
$\Hilb^n\T$ is $|X''|-2$. 
Collecting these results together, we can state the following.
\begin{thm}
In the $1$-parameter family of stability conditions $(\A_0,\mu_t)$ the
moduli functor $\M_t^{(1,2\ell,4-n)}$ has $\lfloor (n-1)/2\rfloor$
walls for all $n>0$ except for $n=5$ when there are $3$ walls. The
highest wall is at $t=\sqrt{n-2}$ and, except
for $|X|=5$, the lowest is at $\sqrt{1+(n+1\mod 2)}$
\end{thm}
So the generating series for the number of walls is 
\[{\frac {{x}^{3} \left(1+ {x}^{2}-{x}^{3}-{x}^{4}+{x}^{5} \right) }{
 \left( 1+x \right)  \left(1-x \right) ^{2}}}
\]

We can extend this to $s$ in the interval $(0,1)$ by observing that
any further destabilizing objects for $L^2\I_X$ with Chern characters
$(r,c\ell,\chi)$ would result in a destabilizing condition of the form
\begin{align*}
0<\chi(2&-s)+(s^2+n-4)(c-2r)=\\
 &-\left(2r- c \right)  \left( s+\frac{1}{2}{\frac {\chi+(n-4)r}{2r-c}}
 \right) ^{2}+2\,\chi+(n-4)c+
\frac{1}{4}{\frac { \left(\chi+(n-4)r \right) ^{2}}{2r-c}}
\end{align*}
Note that $c/r<2$ as the destabilizing object must be a sheaf $K$ in $T_s$
for $0<s<1$ and the kernel of the map $K\to L^2\I_X$ is in $F_s$.
Since we require the centre to be in $(0,1)$ we have
$\chi<-(n-4)r$. But this contradicts the destabilizing inequality.
Combining this with Proposition \ref{p:stabs}(3), we have the following.
\begin{prop}\label{p:nowalls}
For all $n\geq4$, the only walls associated to the Chern character
$(1,2\ell,4-n)$ in the region $0\leq s<2$ are those 
which intersect $s=0$.
\end{prop}
 The situation for $n=3$ is different (see
section 5.1 below).

\section{Projectivity of the Moduli Spaces}
If we number the walls $i=0,\ldots, d=\lfloor (n-3)/2\rfloor$ from the
greatest $t$ downwards then we have $\lfloor (n+1)/2\rfloor$ potential
moduli spaces $M_i$, with $M_0=\Hilb^n\T\times\dT$ (and analogously
for $n=5$).
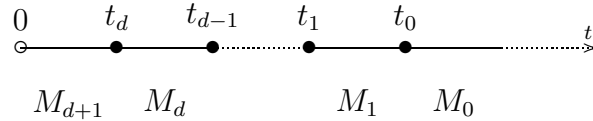
\[\xy\xygraph{*=0{\circ}!{*++!D=0{0}} -_*+!D{\textstyle M_{d+1}} [r]*=0{\bullet}!{*++!D=0{t_d}}
  -_*+!D{\textstyle M_{d}} [r]*=0{\bullet}!{*++!D=0{t_{d-1}}} :@{..} [r]
  *=0{\bullet}!{*++!D=0{t_1}} -_*+!D{\textstyle M_1} [r] *=0{\bullet}!{*++!D=0{t_0}}
  -_*+!D{\textstyle M_0} [r] :@{..>}^(0.9){t} [r] }\endxy\]

\begin{thm}\label{t:main}
For any $t>0$, the moduli space of $t$-stable objects with Chern
character $(1,2\ell,4-n)$ in $\A_0$ is a
smooth complex projective variety for each positive integer $n$.
\end{thm}

The fact that $M_i$ are fine moduli spaces given by smooth varieties
will follow from key results in \cite{Arcara07} 
(generalized a little to cover our case) and we will deal with this in
the next section. We first show that the spaces $M_i$ are
projective. We shall assume in this section that $n\geq4$. The case
$n=3$ will be dealt with as a special case in section 5.1 below.

The trick is to consider the region $t>0$ and $0\leq s<1$ in the set of stability
conditions. Proposition \ref{p:nowalls} tells us that, for a given
$n\geq4$ there are no further walls. The
condition for a wall is given by
\[t^2+\left( s+\frac{n-m-3}{2} \right) ^{2} - \left(\frac{n-m-3}{2}
 \right) ^{2}-(n-2-2\,m)=0\]
corresponding to destabilising sheaves $L\I_{X'}\P_\xhat$ with
$|X'|=m$. The resulting semicircles are illustrated in Figure 1 for
the case $n=10$.
\begin{figure}
\begin{tikzpicture}[scale=1.8]
\draw [help lines] (-0.1,-0.1) grid (2.2,3);
\draw [->] (-0.1,0) -- (2.4,0)  node [anchor=south] {$s$};
\draw [->] (0,-0.1) --  (0,3.3) node [anchor=west] {$t$};
\node[anchor=north] at (1,0) {$1$};
\node[anchor=north] at (0,0) {$0$};
\node[anchor=north] at (2,0) {$2$};
\draw [-] (2,0) -- (2,3.2);
\clip (-0.1,0) rectangle (2.2,3);
\draw (-3.5,0) circle (4.5);
\draw (-3,0) circle (3.873 );
\draw (-2.5,0) circle (3.202);
\draw (-2,0) circle (2.449);
\end{tikzpicture}
\caption{Chamber and walls for $n=10$}
\end{figure}
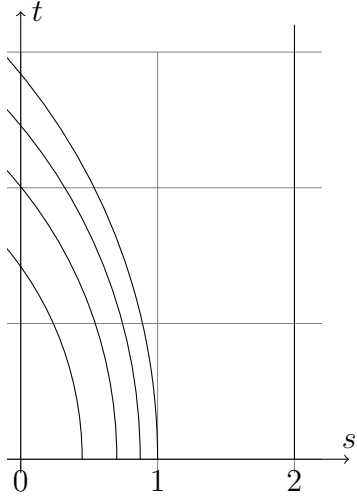
\par
The semicircles intersect the $t=0$ axis in distinct points (as can be
easily checked) and so for each moduli space $M_i$ we can always find
a rational number $s=q_i$ which lies between the $i^{\mathrm{th}}$ and $i+1^{\mathrm{st}}$ wall
on $t=0$. Now let $\Phi_{-q_i}$ be the Fourier-Mukai
transform given by a universal sheaf $\mathbb{E}$ over $\T\times \dT$ whose
restriction $E_\xhat=\mathbb{E}|_{\T\times\{\xhat\}}$ satisfies $\ch(E_\xhat)=(a^2,-ab\ell,b^2)$ where $b/a=q_i$
written in its lowest form. Then $\Phi_{-q_i}(\A_{q_i})=\A_{r_i}$,
where $r_i=c/a$ and $c_1(E_x)=c\ell$, where
$E_x=\mathbb{E}|_{x\times\dT}$ is the restriction to the other factor. Moreover, $e\in
\A_{q_i}$ is $t$-stable for $t\ll 1$ if and only if $\Phi_{-q_i}(e)$ is
$t$-stable for $t\gg0$ in $\A_{r_i}$. Since, our Chern characters
$(1,2\ell,4-n)$ are primitive we know that the moduli space
$M_t^{\Phi_{-q_i}(1,2\ell,4-n)}$ is a fine moduli space of
torsion-free sheaves for $t\gg0$ and is projective provided it is
non-empty. Consequently it will follow that $M_i$ is 
also projective. Since the codimension of the non-torsion-free sheaf
locus in $M_i$ is greater than $n/2-1$ the non-emptyness of
$M_t^{\Phi_{-q_i}(1,2\ell,4-n)}$ will follow from the following.
\begin{prop}
Let $0\leq q< 1$ be a rational number.
If $n>3$ there is some $X\in\Hilb^n\T$ such that
$\Phi_{-q}(L^2\I_X)$ is a torsion-free sheaf in $\A_r=\Phi_{-q}(\A_{q})$.
\end{prop}
Before the proof we recall a few facts and definitions about
Fourier-Mukai transforms. We say that an object $e$ is $\Phi_{-q}$-WIT$_i$
if the cohomologies $\Phi^j_{-q}(e)$ in $\coh_{\dT}$ are zero for all $j\neq
i$. If $a\in\A_q$ then
$\Phi_{-q}^1(A^0)=0=\Phi^{-1}_{-q}(A^{-1})$ (since
$\Phi_{-q}(a)\in\A_r$). We denote the inverse transform by $\hat\Phi_{-q}$.
\begin{proof}
Let $X$ be an arbitrary element of $\Hilb^n\T$.
Observe first that $\ch(E_\xhat\otimes
L^2\I_X)=(a^2,a(2a-b)\ell,(4-n)a^2+b^2-4ab)$. So
$\chi(E_{\xhat}\otimes L^2I_X)<0$  for $n>3$ (in fact, this also works
for $n=3$ for a suitable choice of $a$ and $b$ but this case is not
required). Consider the structure sequence, 
\[0\to L^2\I_X\to L^2\buildrel f\over\to\OO_X\to 0.\]
Now $\Phi_{-q}(L^2)[-1]$ is a sheaf of rank $(2-q)^2a^2$.
Note also that $\Phi_{-q}(\OO_X)[-1]$ is a sheaf. Let
$k=\Phi_{-q}(L^2\I_X)$. Our aim is to show that we can find $X$ so
that $k$ is a torsion-free sheaf. Suppose for a contradiction that $K^{-1}\neq0$. 
\par\noindent\textbf{Claim 1.} \textit{$G=\hat\Phi_{-q}(K^{-1}[1])$ is a
torsion-free sheaf and there is a non-trivial map $G\to L^2\I_X$ which
injects in $\A_q$. Moreover, $\deg(G)>2q\rk(G)$ and $\chi(G)>\rk(G)q^2$.}
\par\noindent
This follows be applying $\hat\Phi_{-q}$ to the triangle
$k\to\Phi_{-q}(L^2)\to\Phi_{-q}(\OO_X)$. We first take cohomology
(in $\coh_{\dT}$) and split the sequence via a sheaf $Q$:
\[\xy\xymatrix@=1ex{0\ar[rr]&&
  K^{-1}\ar[rr]&&\Phi_{-q}(L^2)[-1]\ar[rr]\ar[dr]&&
\Phi_{-q}(\OO_X)[-1]\ar[rr]&&K^0\ar[rr]&&0.\\
&&&&&Q\ar[ur]}\endxy\]
This gives the following short exact sequence in $\A_q$:
\[0\to \hat\Phi_{-q}(K^{-1}[1])\to L^2\to\hat\Phi_{-q}(Q[1])\to 0.\]
Then we see that $\hat\Phi_{-q}(K^{-1}[1])$ is a torsion-free sheaf $G$ and $G\in
T_q$. Since $G$ is $\Phi_{-q}$-WIT$_{-1}$ we must have
$\chi(E_{\xhat}\otimes G)>0$. So $a^2\chi(G)>ab\deg(G)-b^2\rk(G)$. But
$G\in T_q$ so $\deg(G)/\rk(G)>2q$. So we have
$a^2\chi(G)>2qab\rk(G)-b^2\rk(G)=b^2\rk(G)$. 
\par\noindent\textbf{Claim 2.} \textit{$G$ also satisfies
  $\deg(G)<4\rk(G)$.}
\par\noindent
We have a triangle (which is short exact in $\A_q$)
\[ G\to L^2\I_X\to \hat\Phi_{-q}(K^0).\]
Then there is a surjection $L^2\I_X\to\hat\Phi_{-q}^0(K^0)$ in
$\coh_\T$. But $\hat\Phi^0_{-q}(K^0)\in T_q$ and is a torsion
sheaf. Let the dimension of its support be $w$. Then
$c_1(G)-(2-w)\ell=c_1(\hat\Phi_{-q}(K^0))=c_1(\hat\Phi_{-q}^{-1}(K^0))$. But
$\hat\Phi_{-q}^{-1}(K^0)\in F_q$. So $\deg(G)-2(2-w)\leq
2q(\rk(G)-1)<2\rk(G)-2$. Thus $\deg(G)<2\rk(G)+2(1-w)\leq 4\rk(G)$ as required.
\par\noindent\textbf{Claim 3.} \textit{Fix $0<s<1$ (for the $n=5$ case
  also assume $s$ is larger than where the higher rank wall cross $t=0$). There is some $X$ such that
$\Hom(G,L^2\I_X)=0$ for all torsion-free sheaves $G\in T_q$ with
$\ch(G)=(r,c\ell,\chi)$ such that $\chi>rs^2$ and $2r>c>rs$. }
\par\noindent
Consider the numerator of $\mu_t(G)-\mu_t(L^2\I_X)$. This is given by
\[\chi(2-s)+s^2(c-2r)+(n-4)(c-rs)-(2r-c)t^2.\]
But 
\[
\chi(2-s)+s^2(c-2r)+(n-4)(c-rs)>rs^2(2-s)+s^2(rs-2r)=0
\]
So such a $G$ must destabilize in $0<s<1$ and this is impossible
unless $\rk(G)=1$. But then we can pick $X$ so that $\Hom(G,L^2\I_X)=0$
for all $G$ as required.

Returning to the proof we see that $K^{-1}$ must be zero as its
transform cannot map non-trivially to $L^2\I_X$.
\end{proof}

\section{The Surgeries}
It remains to show that $M_i$ are smooth varieties which are fine
moduli schemes representing the appropriate moduli functor \ref{e:modfunct}.
The proof proceeds in exactly the same way as \cite[Theorem
5.1]{Arcara07} but the details are a little different. We first state
a generalization of the Arcara Bertram construction. The details of
the proofs are exactly the same as in \cite{Arcara07} and we omit them.

We state the following in generality for a general Bridgeland stability
condition $(\A,Z)$ given by a fixed abelian subcategory $\A\subset D^b(S)$,
where $S$ is any K3 or abelian surface over $\C$.

\begin{thm}[\cite{Arcara07}]
Fix a Mukai vector $v$ and suppose there is a path $p:\R\to U$ in the
stability manifold for which $\A$ remains fixed. Suppose $M$ is some
fine moduli space of Bridgeland stable objects on 
$S$ which is smooth and proper over 
$\C$ and represents the moduli of $p(t)$-stable objects for $t<0$.
Furthermore, suppose that  $M$ contains a sub-moduli space $P$ whose
objects $a$ satisfy the following conditions  
\begin{enumerate}
\item $a$ becomes unstable for $t>0$.
\item $a$  are represented as short exact sequences
\[0\to e_1\to a\to e_2\to 0\qquad\text{in $\A$}\]
for $e_1\in B_1$ and $e_2\in B_2$, where $B_1$ and $B_2$
are fine moduli spaces of such objects.
\item For all $e_1\in B_1$ and $e_2\in B_2$ and non-trivial extensions 
\[0\to e_2\to b\to e_1\to0\qquad\text{in $\A$}\]
 $b$ is $Z$-stable for all $Z\in U$.
\item \[N:=\dim\Ext^1(e_1,e_2)-1>1\]
for all $e_1\in B_1$ and $e_2\in B_2$.
\item All $a'\in M\setminus P$ are $p(t)$-stable for all $t$.
\end{enumerate}
Then
\begin{enumerate}
\item $P$ is a projective bundle over $B_1\times B_2$ with projective
  space fibres of dimension $N$.
\item There is a smooth proper variety $\operatorname{MF}(M,P)$ which is the Mukai
  flop of $M$ along $P$ and which is a fine moduli space of objects
  which are $p(t)$-stable for all $t>0$.
\end{enumerate}
\end{thm}
Note that the assumptions imply that $\Hom(e_1,e_2)=0=\Hom(e_2,e_1)$
because extensions on either side of the wall are stable for some $t$
and so are simple. Hence, for the abelian surface case,
$N=-\chi(e_1,e_2)-1$. 

The proof is exactly as given in \cite{Arcara07} except that Lemma 5.4
in that paper is not required. This is the only place the particular
choices of $\beta$, $B_1$ and $B_2$ mattered. In fact, the lemma is
unlikely to be true in our cases. The lemma is used to show that the
constructed universal sheaf $\mathcal{U}$ on $M_i$ satisfies
\[\mathcal{U}|_{\T\times P}\cong \mathcal{E}_i\otimes L\]
where $\mathcal{E}_i$ is the universal sheaf corresponding to the
(fine) moduli space $P$ and $L$ is some line bundle pulled back from
$P$. It would then allow us to assume $L$ is
trivial by choice of $\mathcal{U}$. But this is not needed for their argument.

This theorem applies to each of our walls because Proposition
\ref{p:stabs} (3) and (5) implies that the rank 1 walls satisfy the
hypotheses of the theorem as $e_1$ takes the form $L\I_X\P_\xhat$ 
and $e_2$ is a pure torsion sheaf with $c_1(e_2)=\ell$. As we have
already observed (and used) $N>1$ for $n\geq4$. The other hypotheses
are met because no two walls intersect near $s=0$. For the unique rank
2 wall (when $n=5$) we have $e_2=L^{-1}\P_\xhat[1]$ and so is
$t$-stable by Proposition \ref{p:stabs}(1). Finally, $e_1$ is a
$\mu$-stable sheaf of Chern character $(2,\ell,0)$. These are
Fourier-Mukai transforms of pure torsion sheaves with
$c_1=\ell$. These are $t$-stable for all $t$ by Proposition
\ref{p:stabs}(5) again and so $e_1$ must also be $t$-stable (for all
$t$ and, in particular, for the values of $t$ near the wall).

This completes the proof of our main Theorem \ref{t:main}.

\begin{section}{Examples of the Moduli Spaces}
Let us now consider the 
low values of $n$ in more detail. 
\subsection{$\mathbf{n=3}$}
In this case, the only possible value of $m$ is zero and a non-trivial
extension $0\to G\to E\to L\P_\xhat\to 0$ has $G$ with degree $1$. The
argument above proves that this is $t$-stable for $t<1$ and
$t$-unstable for $t>1$. On the other hand, $L^2\P_\xhat\I_Y$ is
$t$-stable for all $t>1$ and all $(Y,\xhat)\in\Hilb^3\T\times\dT$. If
$Y$ is not itself collinear then $L^2\P_\xhat\I_Y$ remains $t$-stable
for all $t$. But, if $Y$ is collinear then $L^2\P_\xhat\I_Y$ is
destabilized by some $L\P_\yhat$. So we have one wall $t=1$ and
consequently two moduli spaces $M_{<1}$ and $M_{>1}$. The latter is
just given by the twisted ideal sheaves. The former has a Zariski open
subset corresponding to non-collinear length $3$ $0$-schemes. The
complement of this is a divisor in $M_{<1}$ and consists of sheaves
with torsion subsheaves of the form $G$ above. In particular, $M_{>1}$
and $M_{<1}$ are birationally equivalent. The existence of $M_{<1}$ as
a fine moduli space will be established in the proof of the theorem below.

Note that $\Ext^1(L\P_\xhat,G)$ has dimension $2$ for all $\xhat$ and
such $G$. This is because $\chi(L,G)=-2$ but $\Hom(L\P_\xhat,G)=0$ for
all $\xhat$ and $G$. Indeed, any such map must factor through a
subsheaf with $c_1=\ell$ and $\chi\leq 0$. Then the kernel is
torsion-free with degree $0$ and $\chi\geq 1$, which is
impossible. The moduli space of such $G$ is isomorphic to
$\T\times\dT$ given by $(x,\xhat)\mapsto \OO_{D_{x}}(1)\P_\xhat$. Then
the space of isomorphisms classes of these sheaves $E$ is a $\PP^1$ bundle
over $\T\times\dT\times\dT$. On the other hand, we can also parametrize the
points $L^2\P_\xhat\I_Y$ where $Y$ is collinear by the dual bundle
(corresponding to $\Ext^1(G,L)\cong \Ext^1(L,G)^*$ under Serre
duality). In particular, the birational map given by identifying the
points corresponding to non-collinear length $3$-subschemes does not
extend to an isomorphism of spaces. Nevertheless, the spaces are isomorphic.

\begin{thm}
The moduli spaces $M_{<1}$ and $M_{>1}$ exist and are isomorphic as
smooth projective varieties.
\end{thm}
\begin{proof}
We know from Geometric Invariant Theory that $M_{>1}$ exists as a fine
moduli space. To show that $M_{<1}$ exists we apply the Fourier-Mukai
transform. This immediately tells us that $M_{<1}$ is isomorphic to
$M^{(-1,2\ell,-1)}_{>1}$. We can give an explicit model for this space
with a universal object as follows.

\vspace{5pt}
\noindent\textbf{Claim 1.} \textit{The points of $M^{(-1,2\ell,-1)}_{>1}$ are
given by objects $e\in \A_0$ such that $E^{-1}\cong L^{-2}\P_\xhat$
for some $\xhat$ and $E^0\in\Hilb^3\T$ such that
$[e]\in\Ext^2(E^0,E^{-1})$ has maximal rank.}

By this last statement we mean that the composite of $[e]$ with any
non-zero map $\OO_x\to E^0$ from a skyscraper sheaf to
$E^0\cong\OO_{Y}$ is also
non-zero. Now the Mukai spectral sequence (as used at the end of the
last section) gives us a long exact sequence of sheaves 
\[0\to \Phi^{-1}(e)\to H_{Y}\tof{f}
\Phi^1(L^{-2}\P_\xhat)\to\Phi^0(e)\to0,\]
where $H_Y$ is the homogeneous bundle which is the Fourier-Mukai
transform of $\OO_Y$.
Note that $\Phi^1(L^{-2}\P_\xhat)$ is a rank $4$ $\mu$-stable vector
bundle. Then $f$ injects precisely when $[e]$ has maximal rank. Then
maximal rank is precisely the condition for $\Phi(e)$ to be a sheaf.

On the other hand, all of the objects of $M_{<1}$ have transforms which
are described by the claim. For the case when $L^2\P_{\yhat}\I_{Y'}$
has non-collinear $Y'$ is given explicitly in \cite{Mac11} Theorem
7.3. The other points of $M_{<1}$ are given as extensions
\[0\to G\to E\to L\P_\yhat\to0,\]
where $G$ is (the direct image of) a line bundle of degree 1 on some
$D_x$. But such $G$ have $\Phi(G)$ of the same form and so applying
$\Phi^*$ we have the  exact sequence
\[0\to\Phi^{-1}(E)\to\hat L^{-1}\P_{\yhat}\tof{g} \hat G\to \Phi^0(E)\to0.\]
But $g$ cannot surject as the kernel must be locally free. Hence, its
image is a torsion sheaf supported on the support of $\hat G$. This
implies that $\Phi^{-1}(E)\cong \hat L^{-2}\P_\xhat$, for some $\xhat$
and $\Phi^0(E)\in\Hilb^3\dT$. This completes the proof of the claim.

\vspace{5pt}
\noindent\textbf{Claim 2.} \textit{The isomorphism class of $e$ is
 independent of $[e]\in \Ext^2(E^0,E^{-1})$. }

This statement is equivalent to saying that the isomorphism type of a
quotient $\widehat{L^{-2}}/H_Y$ is independent of the (injective) map
$H_Y\to \widehat{L^{-2}}$. But this follows because any two such maps
$g$ and $g'$
are equivalent under the composition action of $\Hom(H_Y,H_Y)$ and so
the two quotients $\coker(g)$ and $\coker(g')$ are isomorphic.

Now we see that $M_{>1}^{(-1,2\ell,-1)}$ is given by
$\Hilb^3\T\times\dT$ with universal sheaf
$\pi_{12}^*\OO_{\mathcal{Y}}\otimes\pi_{13}^*\P$ over
$\T\times\Hilb^3\T\times\dT$. In particular, $M_{<1}$ is isomorphic to
$M_{>1}\cong\Hilb^3\T\times\dT$. 
\end{proof}
In fact, the isomorphism can also be given as $\Phi\compo R\Delta$,
where $R\Delta$ is the derived dual functor $\rhom(-,\OO_\T)[1]$. This is
because $R\Delta:M_{>1}\to M_{>1}^{(-1,2\ell,-1)}$. To see this
observe that $R^{-1}\Delta(L^2\I_Y)\cong L^{-2}$ and
$R^{0}\Delta(L^2\I_Y)\cong\OO_{Y}$ and $\OO_{Y}\to
L^{-2}[2]$ must have maximal rank as taking the dual again gives a map
$L^2\to \OO_Y$ which must surject to have come from $L^2\I_Y$. 

Using the calculation in \cite{Mac11} Theorem 7.3 we can write down
the map $M_{>1}\to M_{<1}$ explicitly at a reduced 0-scheme
$Y=\{p,q,y\}$ as
\[\begin{pmatrix}-1&-1&0&-1\\0&-1&-1&-1\\-1&0&-1&-1\\1&1&1&1\end{pmatrix}.\]
thought of as acting on the ``vector'' $(p,q,y,\xhat)$.
In particular, it is not the extension of the birational map $M_{>1}- \to M_{<1}$.

For completeness observe that we have a fourth moduli space
$M^{(-1,2\ell,-1)}_{<1}$. This is the Fourier-Mukai transform space of
$M_{>1}$ and consists of generic points of the sort described in Claim
1 above but a codimension 1 subvariety consists of 2-step complexes
with cohomology $L^{-1}\P_\yhat$ and $G$ where $G$ is a degree 1 line
bundle supported on some $D_x$.

When $0<s<1$ there is a further wall due to destabilizing objects of
the form $\Phi(L^{-2})[1]$. This corresponds to a ``codimension 0''
surgery. It is an exercise to check that there are no further
destabilising objects for $-1<s<1$ and so the chamber and wall
structure is as illustrated in Figure 2. Once we cross this additional
wall the moduli space consists of objects $e$ of the form
\[H_{\tilde Y}[1]\to e\to \Phi^{1}(L^{-2}\P_\xhat)\]
Then the Fourier-Mukai transform under $\Phi$ of this space is exactly
$\Hilb^3\dT\times\T$ given by sheaves of the form
$L^{-2}\P_x\I_{\tilde Y}$.
\begin{figure}
\begin{tikzpicture}[scale=1.3]
\draw [help lines] (-2.1,-0.1) grid (2.2,2);
\draw [->] (-2.1,0) -- (2.4,0)  node [anchor=south] {$s$};
\draw [->] (0,-0.1) --  (0,2.3) node [anchor=west] {$t$};
\node[anchor=north] at (2,0) {$s=1$};
\node[anchor=north] at (0,0) {$s=0$};
\node[anchor=north] at (-2,0) {$s=-1$};
\clip (-2.2,0) rectangle (2.2,2);
\draw (0,0) circle (2);
\draw (0.5,0) circle (0.5);
\end{tikzpicture}
\caption{Chamber and walls for $n=3$}
\end{figure}
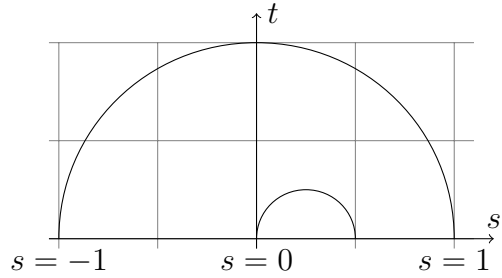

\subsection{$\mathbf{n=4}$}
Again there is only one wall, this time at $t=\sqrt{2}$. 
Just as for the length $3$ case, it is the collinear length $4$
0-schemes which correspond to non $t$-stable sheaves $L^2\I_Z$ as $t$
crosses the wall. These live in a codimension $2$ subvariety and so we
can use the Arcara Bertram argument from \cite{Arcara07} to construct
our moduli space $M_{<\sqrt{2}}$ as a Mukai flop of $M_{>\sqrt{2}}$. This is explained
more fully in the next section.

In this case, the Fourier-Mukai transform gives us an isomorphism
$M_{<\sqrt{2}}\cong M_{>1/\sqrt{2}}^{(0,2\ell,-1)}$ which consists of
pure torsion sheaves of rank $1$ and degree $3$ supported on a translate of
a divisor in the linear 
system $|2\ell|$. In particular, the moduli space $M_{<\sqrt{2}}$ is projective.
The points of $M^{(0,2\ell,-1)}_{>1/\sqrt{2}}$ are harder to describe because this linear
system has singular and reducible elements. For the Chern character
$(0,2\ell,-1)$ there is exactly one wall at $t=1/\sqrt{2}$ and we need
to glue in Fourier-Transforms of $L^2\I_Z$ (and their flat twists) corresponding to
collinear $Z$. These are computed in \cite{Mac11}. The objects are
2-step complexes with cohomology $L^{-1}\P_{-x}$ and $L\P_{-x+\Sigma
  Z}\I_{2x-\Sigma Z}$, where $Z\subset D_x$. 

This should be compared
with the situation in \cite{Arcara07}. The nearest such space (in
their notation) is
$H=2\ell$ and we take $\A_2\cong (-\otimes L)(\A_0)$. The
corresponding Chern character is $(0,2\ell,4)$ rather than
$(0,2\ell,3)$ as in our case.  Of course, $H$
is reducible and so their construction does not apply. But
nevertheless, we obtain analogous data. There is a wall at $t=1/2$ and
glue in 2-step complexes whose $-1$ cohomology is (a twist of)
$L^{-1}$ and whose $0$th cohomology is an extension of (a twist of) $L\I_y$ by
$\OO_z$. 

\subsection{$\mathbf{n=5}$}
\setlength{\tabcolsep}{20pt}
The length $5$ case is special because of the higher rank wall which
intersects $s=0$. There are four moduli spaces corresponding to the 3
walls. The configuration is illustrated in Figure 3.
\begin{figure}[h!]
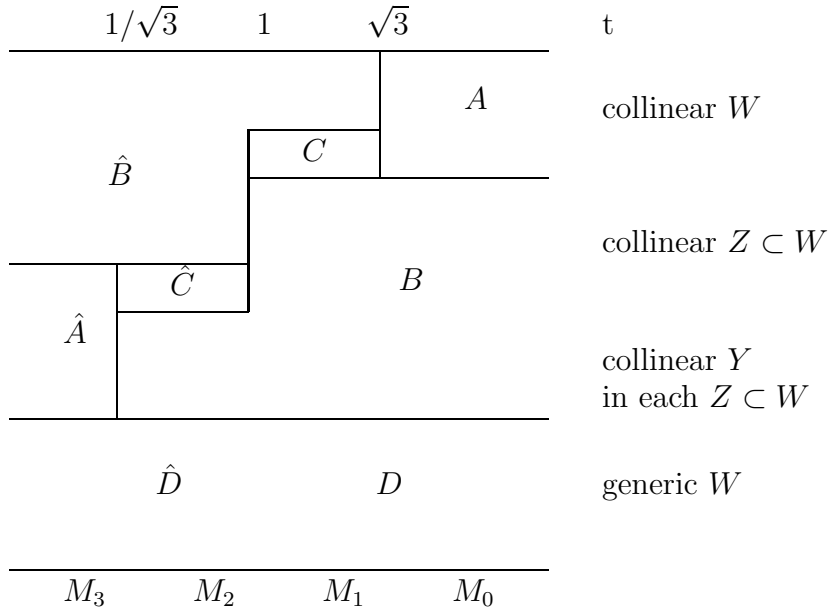

\begin{tabular}{c|c|c|cl}
\multicolumn{4}{c}{$1/\sqrt{3}$\hspace{0.9cm} $1$\hspace{1.1cm} $\sqrt{3}$}\rule{0.6cm}{0pt}\bigstrut[b]&t\\
\cline{1-4}
\multicolumn{3}{c|}{\multirow{3}{2.3cm}[-2ex]{$\hat B$}} &
\rule{0pt}{0.9cm}\multirow{2}{0pt}[2.0ex]{$A$}&{collinear $W$}\\
\cline{3-3}
\multicolumn{2}{c|}{} &$C$\bigstrut\\
\cline{3-4}
\multicolumn{2}{c|}{}&\multicolumn{2}{c}{\multirow{3}{0pt}{$B$}}\rule{0pt}{1cm}&
{collinear $Z\subset W$}\\
\cline{1-2}
\multirow{2}{0pt}[-2ex]{$\hat A$}&$\hat C$\bigstrut&\multicolumn{2}{c}{}\\
\cline{2-2}\\
&\multicolumn{3}{c}{}&{\parbox[b]{3cm}{collinear $Y$\\ in each $Z\subset W$}}\\
\cline{1-4}
\multicolumn{4}{c}{$\hat D$\hspace{1in}$D$}&{generic $W$}\rule[-1cm]{0pt}{2cm}\\
\cline{1-4}
\multicolumn{4}{c}{$M_3$\hspace{1cm} $M_2$\hspace{1cm} $M_1$\hspace{1cm} $M_0$\bigstrut[t]}
\end{tabular}
\caption{Diagram of surgeries for $n=5$}
\end{figure}
 The vertical
lines indicate the walls. The horizontal lines indicate strata in each
moduli space. The letters $A$, $B$, $C$, $D$  indicate sheaves (or 2-step complexes) of a
particular type and their corresponding hatted letters are the
Fourier-Mukai transformed spaces.  To the right of a wall in regions A, B and D
we have torsion-free sheaves characterised by the geometric property
indicated. The codimensions of the spaces are $A$ is codimension $3$
in $M_0$, $B$ is codimension $2$ in $M_0$  and $C$ is codimension 4 in
$M_1$. In particular, $C$ is codimension $1$ in the replacement for
$A$ in $M_1$.

\end{section}

\section*{Acknowledgements}
The authors would like to thank Arend Bayer, Tom Bridgeland, Lothar G\"ottsche, Daniel Huybrechts
and especially Aaron Bertram for useful hints. The ``sliding down the
wall'' trick we use in section 3 is due to Bertram and his co-workers. They would also like to
thank the organisers of the Moduli Spaces Programme and the Isaac
Newton Institute in Cambridge for their hospitality while this work
was carried out. This work forms part of the second author's PhD thesis.

\bibliographystyle{alpha}
\bibliography{p56}

\end{document}